\documentclass[12pt]{article}
\usepackage[top=1in, bottom=1in, left=1.2in, right=1.2in]{geometry}
\usepackage{graphicx, amsfonts,amssymb,amsthm,amsbsy,latexsym,amscd,amsmath,dsfont,euscript,enumerate,manfnt, marvosym,verbatim, calc,mathrsfs,marvosym,textcomp, color,xcolor,tikz-cd,setspace,stmaryrd,upgreek,bm} % Required for inserting images

\usepackage{pgf,tikz,amsmath}
\usepackage{tikzit}
\tikzstyle{Doty}=[fill=black, draw=black, shape=circle, tikzit shape=circle]

\tikzstyle{new edge style 0}=[-, draw=red]

\usepackage{blkarray}
\usepackage{caption}
\usepackage{enumerate}
\usepackage[pagebackref,colorlinks=true,
linkcolor=blue,
urlcolor=red,colorlinks,
citecolor=green]{hyperref}
\usepackage{cleveref}
\usepackage{parskip}
\usepackage{setspace}
\makeatletter
\def\thm@space@setup{%
  \thm@preskip=10pt % space before theorem
  \thm@postskip=6pt % space after theorem (optional)
}
\makeatother

\usepackage{refcount}

\newtheorem*{theorem*}{Theorem}
\newtheorem{theorem}{Theorem}

\newtheorem{lemma}{Lemma}[section]

\crefname{lemma}{Lemma}{Lemmas}
\crefname{corollary}{Corollary}{Corollaries}

\Crefname{lemma}{Lemma}{Lemmas}
\Crefname{corollary}{Corollary}{Corollaries}

\crefalias{corollary}{corollary}

\newtheorem{claim}{Claim}[theorem]

\crefname{claim}{claim}{claims}
\Crefname{claim}{Claim}{Claims}

\crefname{subclaim}{subclaim}{subclaims}
\Crefname{subclaim}{Subclaim}{Subclaims}

\theoremstyle{definition}

\newtheorem*{question}{Question}
\newtheorem{remark}{Remark}

\newtheorem*{acknowledgement}{Acknowledgement}

\title{On Ramsey number of $K_{2,n}$ versus even cycles}
\author{Abisek Dewan\footnote{Department of Mathematics and Statistics, IISER Kolkata, India (ad22rs069@iiserkol.ac.in)}\and Sayan Gupta\footnote{School of Mathematical Sciences, NISER Bhubaneswar (An OCC of Homi Bhabha National Institute, Mumbai), India (sayan.gupta@niser.ac.in)}\and Rajiv Mishra\footnote{Department of Mathematics and Statistics, IISER Kolkata, India (rm20rs017@iiserkol.ac.in)}}
\date{ }

\begin{document}

\maketitle

\begin{abstract}
    For graphs $G$ and $H$, the Ramsey number $R(G,H)$ is the smallest integer $N$ such that every graph $\Gamma$ on $N$ vertices contains $G$ or its complement $\overline{\Gamma}$ contains $H$ as a subgraph. In graph Ramsey theory, the star-cycle Ramsey number is well-studied throughout the years. Whereas the Ramsey number of $K_{2,n}$ versus cycle is challenging to determine due to increased structural complexity. In this article, we have obtained an exact value of the Ramsey number $R(K_{2,n}, C_{m})$ for even $m\in [n, 2n-4008]$ and $n\geq 4516$. In particular, we show that $$R(K_{1,n}, C_{m})= R(K_{2,n}, C_{m})$$ for all even $m\in [n, 2n-4008]$ and $n\geq 4516$. This leads to an interesting question: For fixed $t$, does there exist $n_0(t)\in \mathbb{N}$ such that $R(K_{1,n}, C_m)=R(K_{t,n}, C_m)$ for all $n \geq n_0(t)$ and for a given range of even $m$?   
\end{abstract}



\section{Introduction}

In this article, we consider finite, simple and undirected graphs only.
We start by introducing some standard  notations and terminologies. Let $G$ be a graph with vertex set $V(G)$ and edge set $E(G)$. For a vertex $v \in V(G)$, the neighborhood of $v$ is the set of all vertices adjacent to $v$ and it is denoted by $N_G(v)$. The cardinality of this set is called the degree of a vertex $v$, and is denoted by $d_G(v)$.  The minimum and maximum degrees of $G$ are denoted by $\delta(G)$ and $\Delta(G)$, respectively. The subgraph of $G$ induced by a subset $X \subseteq V(G)$ is denoted as $G[X]$. The circumference $c(G)$ is the length of a longest cycle in $G$, whereas the girth $g(G)$ is the length of a shortest cycle in $G$.

A cycle containing all vertices of $G$ is called a Hamiltonian cycle, and a graph containing such a cycle is said to be Hamiltonian graph. A graph $G$ is said to be pancyclic if it contains cycles of all lengths from 3 to $|V(G)|$, and weakly pancyclic if it contains cycles of all lengths between its girth and circumference. Similarly, a bipartite graph is called bi-pancyclic if it contains cycles of all even lengths from $4$ to $|V(G)|$, and weakly bi-pancyclic if it contains cycles of all even lengths between its girth and circumference. A graph $G$ is said to be $k$-connected if it remains connected after the removal of any set of fewer than $k$ vertices. The maximum integer $k$ for which $G$ is $k$-connected is called the connectivity of $G$, denoted by $\kappa(G)$.

For any two graphs $G$ and $H$, the \emph{Ramsey number}, denoted by $R(G, H)$, is the minimum positive integer $N$ such that for any graph $\Gamma$ on $N$ vertices either $\Gamma \supseteq G$ or $\overline{\Gamma}\supseteq H$ as subgraphs. Cycles and stars have been well studied in graph Ramsey theory since the early 1970s. The following is the well-known result of the star-cycle Ramsey number by Lawrence \cite{SLLawrence1973}, a detailed proof of which can be found in \cite{parsons1978ramsey}.

\begin{theorem}[Lawrence \cite{SLLawrence1973}]
    \begin{align*}
R(K_{1,n},C_{m})&=\left\{\begin{array}{lll}
		2n+1& \textnormal{ if  }&  m \textnormal{ is odd and } m\leq 2n-1\\
        m &\textnormal{  if }& m\geq 2n.
    \end{array}\right.  
\end{align*}
\end{theorem}

Zhang, Broersma, and Chen \cite{zhang2016narrowing} narrowed the gap of star-cycle Ramsey number by finding the exact value of $R(K_{1,n}, C_{m})$ for even $m\in [3n/4+1, 2n]$. In 2023, Allen et al. \cite{allen2023ramsey} filled the remaining gap of the star-cycle Ramsey number by obtaining an exact formula of $R(K_{1,n}, C_{m})$ for large even $m$. 

\begin{theorem}[Zhang et al. \cite{zhang2016narrowing}]\label{thm:zhang}
    \begin{align*}
R(K_{1,n},C_{m})&=\left\{\begin{array}{lll}
		2n& \textnormal{ if  }&  m \textnormal{ is even and } n<m\leq 2n\\
        2m-1 &\textnormal{  if }& m \textnormal{ is even and } \frac{3n}{4}+1\leq m\leq n.
    \end{array}\right.  
\end{align*}
\end{theorem}
As compared to the star-cycle Ramsey number, $K_{2,n}$ versus cycle has been less studied due to the structural complexity of $K_{2,n}$. Here, the notion of complexity refers to the fact that a $K_{2,n}$-free graph $G$ does not provide straightforward information on the minimum degree of the complement graph $\overline{G}$, unlike $K_{1,n}$-free graphs. In \cite{gupta2025study}, author has studied $R(K_{2,n}, C_{m})$ for odd $m$ and $n\geq 2m+499$. In this paper, we focus on finding the exact Ramsey number of $K_{2,n}$ versus even cycles. We have established the exact value of $R(K_{2,n}, C_{m})$ for even $m$ where $m\in [n,2n-4008]$. Our first main result is as follows.

\begin{theorem}\label{thm:main_result11}
    For $n\geq 4511$ and even $m\in [n+1, 2n-4008],$
    $$R(K_{2,n},C_m)=2n.$$
\end{theorem}
The constant term $4511$ in the above result comes naturally from our computation. We also believe that the upper bound of $m$ could be improved to $2n-1$ for sufficiently large $n$. We have also studied the exact value of $R(K_{2,n},C_m)$ for $m=n$ where $m$ is even.

\begin{theorem}\label{thm:main_result22}
    For even $n\geq 4516$, we have $$R(K_{2,n},C_n)=2n-1.$$
\end{theorem}

\begin{remark}
 It is worth noting that for $m \in [n, 2n-4008]$, we have $R(K_{2,n}, C_m) = R(K_{1,n}, C_m)$ whenever $m$ is even and $n$ is sufficiently large. On the other hand, by \Cref{thm:zhang}, $R(K_{1,4},C_4)=7=2n-1$ for $n=4$, whereas $R(K_{2,4},C_4)=9>2n-1$ \cite{harborth1996some}. These examples naturally lead to the following question, which might be interesting to explore.
\end{remark}

\begin{question}
For fixed $t$, does there exist $n_0(t)\in \mathbb{N}$ such that 
\[
R(K_{1,n}, C_m)=R(K_{t,n}, C_m)
\]
for all $n \geq n_0(t)$ and a given range of even $m$?
\end{question}

\section{Preliminary Results}
 
In this section, we collect several known results that will be used throughout the paper. We begin with the classical result of Dirac on the circumference of a $2$-connected graph $G$.

\begin{lemma}[Dirac \cite{dirac1952some}]\label{Dirac}
    If $G$ is a $2$-connected graph with minimum degree $\delta(G)\geq k$, then $G$ contains a cycle of length at least $2k$.
\end{lemma}

\begin{lemma}[Zhang et al. \cite{zhang2014ramsey}]\label{Thm:Zhang_longest_cycle}
    Let $C$ be the longest cycle of a graph $G$. For any $u,v\in V(G)\setminus V(C)$, $|V(C)|\geq 2|(N_G(u)\cup N_G(v))\cap V(C)|-2$.
\end{lemma}
The following lemma is one of the key results used in the proof of our main results. 

\begin{lemma}[Wei \cite{wei1997longestcycle3connected}]\label{Thm:circumference_3connectedGraphs}
    Let $G$ be a $3$-connected graph on $n$ vertices such that $|N_G(u)\cup N_G(v)|\geq s$ for any non-adjacent vertices $u,v\in V(G)$. Then $c(G)\geq \min\{n,3s/2\}$.
\end{lemma}

Pancyclic and weakly pancyclic graphs are often useful in results related to cycle Ramsey numbers. Hamiltonian graphs with sufficiently large degree conditions contain cycles of almost all lengths. Here, we recall some related results concerning Hamiltonicity and pancyclicity.

\begin{lemma}[Nash-Williams \cite{nash1971Hamiltonian}]\label{Thm:Nash-william}
    Let $G$ be a $2$-connected graph of order $n$ with $\delta(G)\geq \max\{(n+2)/3,\alpha(G)\}$. Then $G$ is Hamiltonian.
\end{lemma}

\begin{lemma}[Bondy \cite{bondy1971pancyclic}]\label{Thm:pancyclic-Bondy1971}
    If a graph $G$ of order $n$ has minimum degree 
$\delta(G)\geq n/2$, then $G$ is pancyclic, or $n = 2r$ and $G = K_{r,r}$.
\end{lemma}

\begin{lemma}[Williamson \cite{williamson1977panconnected}]\label{Thm:panconnected}
    Let $G$ be a graph of order $n\geq 4$ such that $\delta(G)\geq (n+2)/2$ then $G$ is panconnected.
\end{lemma}

\begin{lemma}[Hu et al. \cite{hu2015bi-weakly}]\label{Thm:Bi-pancyclic_Hu}
    If $G = (V_1, V_2, E)$ is a bipartite graph with minimum degree at least $n/3 + 4$, where $n = \max \{|V_1|, |V_2|\}$, then
$G$ is a weakly bi-pancyclic graph with girth $4$.
\end{lemma}

\begin{lemma}[Brandt et al. \cite{brandt1998weakly}]\label{Thm:Weakpan-Brandt_nby4_250_1}
    Let $G$ be a $2$-connected non-bipartite graph of order $n$ with minimum degree $\delta(G) \geq  n/4 + 250$. Then $G$ is weakly pancyclic unless $G$ has odd girth $7$, in which case it has every cycle from $4$ up to its circumference except the $5$-cycle.
\end{lemma}

The following result proves that a graph of sufficiently large order with a high degree condition contains either a triangle or a four-cycle. For the detailed proof, one can refer to \cite{gupta2025study}.

\begin{lemma}[Lin et al. \cite{lin2021large}, Gupta \cite{gupta2025study}] \label{thm:upperboundgirth}
 Let $G$ be a graph on $n$ vertices with $\delta(G)\geq cn$ and $n\geq 10/c^2$ for some $c>0$, then $g(G) \leq  4$.
\end{lemma}

Combining  \Cref{thm:upperboundgirth,Thm:Weakpan-Brandt_nby4_250_1}, we obtain the following lemma.

\begin{lemma}\label{Thm:Weakpan-Brandt_nby4_250}
    Let $G$ be a $2$-connected non-bipartite graph of order $n\geq 160$ with minimum degree $\delta(G) \geq  n/4 + 250$. Then $G$ is weakly pancyclic with the girth at most $4$ unless $G$ has odd girth $7$, in which case it has every cycle from $4$ up to its circumference except the $5$-cycle.
\end{lemma}

\section{Main Results}
Before proceeding to the main results, we introduce a lemma that guarantees a long cycle in a $2$-connected graph under a suitable neighborhood condition. This lemma will serve as one of the key tools in the proofs of our main results. 

\begin{lemma}\label{lem:cycle_lemma}
Let $G$ be a $2$-connected graph on $n$ vertices such that
\[
|(N_G(u)\cup N_G(v))\setminus\{u,v\}|\geq k
\]
for all $u,v\in V(G)$. If $c(G)\geq n-k$, then
\[
c(G)\geq \min\{2k-2, n-1\}.
\]
\end{lemma}
\begin{proof}Let $C_\ell$ be a longest cycle in $G$ of length $\ell$, and let $X := V(G)\setminus V(C_\ell)$. 
If $|X|\leq 1$, then $\ell \geq n-1$, and we are done. Hence, assume that $|X|\geq 2$.
Suppose first that $G[X]$ is an empty graph. Then for every $u,v\in X$, we have
\[
|(N_G(u)\cup N_G(v))\cap V(C_\ell)|\geq k.
\]
Therefore, by \Cref{Thm:Zhang_longest_cycle}, it follows that $\ell\geq 2k-2$, as desired. Thus, we may assume that $G[X]$ is not an empty graph. We now claim the following:

\textbf{Claim. }\emph{There exists a longest path $P_s = x_1 x_2 \ldots x_s$ in ${G}[X]$ such that there exist distinct vertices $c_1, c_s \in V(C_\ell)$ with $x_1 \sim c_1$ and $x_s \sim c_s$.}

\emph{Proof of Claim.}   
Let $P_s = x_1 x_2 \ldots x_s$ be a longest path in ${G}[X]$. Clearly, at least one of $x_1$ or $x_s$ has a neighbor in $V(C_\ell)$, otherwise 
\begin{eqnarray*}
    |(N_G(x_1)\cup N_G(x_s))\setminus\{x_1, x_s\}|&\leq& n-2-\ell\\
    &\leq& n-2-(n-k)\\
    &=& k-2,
\end{eqnarray*}
a contradiction. Therefore without loss of generality, assume that $x_1$ is adjacent to some $c_1 \in V(C_\ell)$. If $x_s$ also has a neighbor in $V(C_\ell)\setminus\{c_1\}$, then we are done; hence assume that $x_s$ has no neighbor in $V(C_\ell)\setminus\{c_1\}$.

Since $P_s$ is a longest path, $x_s$ cannot be adjacent to any vertex in $X \setminus V(P_s)$. As ${G}$ is $2$-connected, $x_s$ must be adjacent to at least two vertices of the path; in particular, suppose that $x_s$ is adjacent to $x_{s'}$ for some $s' \le s-2$. Note that $x_{s'+1}$ must have a neighbor in $V(C_\ell)\setminus\{c_1\}$, since otherwise 
\begin{eqnarray*}
    |(N_G(x_s)\cup N_G(x_{s'+1}))\setminus\{x_s, x_{s'+1}\}|&\leq& n-2-(\ell-1)\\ 
    &\leq& k-1,
\end{eqnarray*}
again a contradiction. Now consider the path
\[
P'_s = x_1 \ldots x_{s'}\, x_s\, x_{s-1} \ldots x_{s'+1}.
\]
Note that this is also a path of length $s$ and both endpoints of $P'_s$ have neighbors in $V(C_\ell)$, completing the proof of the claim.

Now assume $P_s=x_1 x_2\ldots x_s$ be a longest path in ${G}[X]$ such that $x_1$ and $x_s$ each have at least one neighbor in $V(C_\ell)$. Since $x_1$ and $x_s$ can be adjacent only to the vertices on the the $P_s$ in $G[X]$, therefore we have
\begin{equation}\label{eqn:KnbronC_l}
    |(N_{G}(x_1) \cup N_{G}(x_s))\cap V(C_{\ell})|\geq k-s.
\end{equation}
We can choose two non-empty sets $A\subseteq N_{G}(x_1) \cap V(C_{\ell})$ and $B\subseteq N_{G}(x_s) \cap V(C_{\ell})$ such that $A\cap B=\phi$ and $|A|+|B|\geq k-s$, by \Cref{eqn:KnbronC_l}. 
 We can define pairwise disjoint subsets
\[
A_1, B_1, \dots, A_q, B_q \subseteq V(C_{\ell})
\]
such that each \( A_i \) and \( B_i \) induces a set of consecutive vertices on \( C_\ell \), the endpoints of \( A_i \) belong to \( A \), the endpoints of \( B_i \) belong to \( B \), and \( A_i \cap B = \phi \), \( B_i \cap A = \phi \) for all \( i \). Let $a_i=|A_i\cap A|$ and $b_i=|B_i\cap B|$ for each $i\in \{1,\ldots,q\}$. Therefore, $\sum_{i=1}^qa_i=|A|$ and $\sum_{i=1}^qb_i=|B|$.

Note that no two elements of $A$ or $B$ can be consecutive on $V(C_{\ell})$ otherwise we have a larger cycle. Thus, $|A_i|\geq 2a_i-1$ and $|B_i|\geq 2b_i-1$ for all $i$. Similarly, for any $a\in A$ and $b\in B$, the circular distance between $a$ and $b$ on $C_{\ell}$ is at least $s+1$.
 Therefore,  
\begin{eqnarray*}
    \ell &\geq& \sum_{i=1}^{q}\left(|A_i|+|B_i|\right)+2qs\\
    &\geq& \sum_{i=1}^q
\left[\left(2a_i-1\right)+\left(2b_i-1\right)\right]+2qs\\
&\geq& 2(|A|+|B|)-2q+2qs\\
&\geq& 2(k-s)+2q(s-1)\\
&\geq& 2k-2 \text{  }\;\;\;\; (\text{since }q\geq 1 \text{ and } s\geq 1).
\end{eqnarray*}
\end{proof}

\setcounter{theorem}{\getrefnumber{thm:main_result11}-1}
\begin{theorem}
For $n\geq 4511$ and even $m\in [n+1, 2n-4008],$
    $$R(K_{2,n},C_m)=2n.$$
\end{theorem}

\begin{proof}
To prove the lower bound, consider the graph $G=K_{n-1,n-1}\sqcup K_1$. Clearly, $G\nsupseteq K_{2,n}$ and note that any cycle in $\overline{G}$ has order at most $n$. Therefore, $\overline{G}\nsupseteq C_m$ for any $m\geq n+1$.

 To prove the upper bound, we first let $n\geq 4511$ and fix the parameter as $m\in [n+1,2n-4008]$ with $m$ even. Suppose on contrary, there exists a graph $G$ on $2n$ vertices such that $G\nsupseteq K_{2,n}$ and $\overline{G}\nsupseteq C_m$. We divide the proof into two cases depending on the minimum degree of $\overline{G}$. 

\textbf{Case 1. } $\delta(\overline{G})\geq \lceil \frac{n}{2}\rceil+250.$

First, we prove that $\overline{G}$ is $2$-connected. If not, let  $\kappa(\overline{G})=0$. Then we can partition $V(\overline{G})$ as $V(\overline{G})=A\sqcup B$, where $A$ is the set of vertices of the smallest connected component of $\overline{G}$. Since $|A|\geq \delta(\overline{G})+1>2$ and $|B|\geq n$, we get a $K_{2,n}$ in $G$, a contradiction. Thus $\kappa(\overline{G})\neq 0$. Using a similar argument, it also follows that $\kappa(\overline{G})\neq 1$.

Since $G\nsupseteq K_{2,n}$, for all $u,v\in V(\overline{G})$, we have
\begin{equation}\label{eqn:K_2nFree_equation}
    |(N_{\overline{G}}(u)\cup N_{\overline{G}}(v))\setminus\{u,v\}|\geq n-1.
\end{equation}
Also, since $\overline{G}$ is $2$-connected and $\delta(\overline{G})\geq \lceil {n}/{2}\rceil+250$, therefore by \Cref{Dirac}
\begin{equation}\label{eqn:mindegree_first}
    c(\overline{G})\geq n+500>2n-(n-1).
\end{equation}
Together with \Cref{eqn:K_2nFree_equation}, \Cref{eqn:mindegree_first} and \Cref{lem:cycle_lemma}, we get
\begin{equation}\label{eqn:mindegree_second}
    c(\overline{G})\geq 2n-4.
\end{equation}

Next, we prove that $\overline{G}$ cannot be bipartite. Suppose, on contrary, $\overline{G}$ is bipartite. Consider the bipartition $V(\overline{G})=A\sqcup B$ such that $|A|\leq |B|$. Clearly $n\leq |B|\leq n+1$, since if $|B|\geq n+2$ then $G\supseteq K_{2,n}$. Since, we have 
$$\delta(\overline{G})\geq \bigg\lceil \frac{n}{2}\bigg\rceil +250\geq \frac{n+1}{3}+4,$$
and also, $\overline{G}$ is $2$-connected, therefore by \Cref{Thm:Bi-pancyclic_Hu}, $\overline{G}$ is weakly bi-pancyclic of girth 4. This implies $\overline{G}\supseteq C_m$ as we have  $c(\overline{G})\geq 2n-4$, a contradiction.

So far we have proved that $\overline{G}$ is a 2-connected, non-bipartite graph with circumference $c(\overline{G})\geq 2n-4$ such that 
$$\delta(\overline{G})\geq \bigg\lceil\frac{n}{2}\bigg\rceil +250\geq  \frac{|V(\overline{G})|}{4}+250.$$
Therefore, by \Cref{Thm:Weakpan-Brandt_nby4_250} $\overline{G}\supseteq C_m$ for all $m\in \{n+1, \ldots, 2n-4\}$, a contradiction. Hence, this case is not possible.

\textbf{Case 2. } $\delta(\overline{G})\leq \lceil \frac{n}{2}\rceil+249.$

In this case, we have
\begin{equation*}
    \Delta(G)\geq 2n-1-\bigg\lceil\frac{n}{2}\bigg\rceil-249
    =\bigg\lfloor\frac{3n}{2}\bigg\rfloor-250.
\end{equation*}
Consider a vertex $v$ such that $|N_G(v)|\geq \lfloor3n/2\rfloor-250$ and choose $X\subseteq N_G(v)$ such that $|X|=\lfloor3n/2\rfloor-250$ and let $Y:=V(G)\setminus (X\cup \{v\})$.
Note that for all $x\in X$, we have $|N_{G}(x)\cap X|\leq n-1$, this implies 
\begin{equation}\label{eqn:claim3_first}
    |N_{\overline{G}}(x)\cap X|\geq \bigg\lfloor \frac{n}{2}\bigg\rfloor-250\;\;\;\;\text{for all }x\in X.
\end{equation}
 Similarly 
 \begin{equation}\label{eqn:claim3_second}
     |N_{\overline{G}}(y)\cap X|\geq\bigg\lfloor \frac{n}{2}\bigg\rfloor-249 \;\;\;\;\text{for all }y\in Y.
 \end{equation}
 Consider $Y'\subseteq Y$ such that $|Y'|=\lceil n/2\rceil-1752$ and take $X'=X\sqcup Y'$. Now consider the graph $\overline{G}[X']$ on $2n-2002$ vertices. By \Cref{eqn:claim3_first} and \Cref{eqn:claim3_second}, we have 
 \begin{equation}\label{eqn:min_degree_Comp_G[X']}
     \delta(\overline{G}[X'])\geq \bigg\lfloor\frac{n}{2}\bigg\rfloor  -250\geq\frac{|X'|}{4}+250.
 \end{equation}
 We now aim to apply \Cref{Thm:circumference_3connectedGraphs} on $\overline{G}[X']$. For that, we need to prove the following claim.

  \begin{claim}\label{sub:claim:2a:thm1}
     $\overline{G}[X']$ is $3$-connected.
 \end{claim}
\emph{Proof of $\Cref{sub:claim:2a:thm1}$.} \textbf{Case 3.1(a).} First suppose $\overline{G}[X']$ is disconnected. Note that $\overline{G}[X']$ can not have more than two connected components since if it has more than two connected components, then the smallest connected component has size at most $(2n-2002)/3$. Since each component contains at least $ \delta(\overline{G}[X'])+1$ many vertices, we can choose two vertices from the smallest component such that they will have at least $(2n-2002)- (2n-2002)/3=(4n-4004)/3\geq n$ common non-neighbors in $\overline{G}[X']$ for $n\geq 4004$, a contradiction. Therefore consider the partition $X'=A\sqcup B$ with $|A|\leq |B|$. Note that $|A|\geq \delta(\overline{G}[X'])+1\geq 2$ and hence $|B|\leq n-1$, otherwise we have two vertices having $n$ non-neighbors in $\overline{G}$, but this implies $|A|\geq n-2001$. Further, we revise the bounds as follows 
\begin{eqnarray*}
    n-2000 \leq &|A|&\leq n-1001\\
    n-1001\leq &|B|&\leq n-2,
\end{eqnarray*}

since if $|B|= n-1$, then we can choose two vertices from $A$, both non-adjacent to $v$ in $\overline{G}$ such that they have at least $|B|+|\{v\}|= n$ common non-neighbors in $\overline{G}$, a contradiction.  We can choose such vertices in $A$ since at least $\lfloor{3n}/{2}\rfloor-250$ many non-neighbors of $v$ in $\overline{G}$ are distributed inside $A\sqcup B$ and $|B|= n-1$.
Now, we consider 
\begin{eqnarray*}
    |A|&=&n-1001-k,\\
    |B|&=&n-1001+k, \;\;\;\; \text{ for some } 0\leq k \leq 999.
\end{eqnarray*}

As $\delta(\overline{G}[A])=\delta(\overline{G}[X'])\geq \lfloor n/2\rfloor -250 > (|A|+2)/2$, therefore by \Cref{Thm:panconnected}, $\overline{G}[A]$ is panconnected. With such a degree condition, we can also conclude that diam$(\overline{G}[A])\leq 2$. Now we claim that $\overline{G}[B]$ is Hamiltonian. We have already seen that $\overline{G}[B]$ is not disconnected. Now suppose $\overline{G}[B]$ has a cut vertex, say $w$, and let $B\setminus\{w\}=B_1\sqcup B_2$. Any two vertices in the smallest component have at least $(n-1001-k)+(n-1002+k)/2\geq n$ common non-neighbors in $\overline{G}[X']$ if $n\geq 4003$, a contradiction. Therefore $\overline{G}[B]$ is 2-connected. 
 Now if the independence number $\alpha(\overline{G}[B])\geq \lfloor n/2\rfloor -249$ then any two vertices from the largest independent set of $\overline{G}[B]$ have at least $n-1001-k+\lfloor n/2\rfloor -251\geq n$ common non-neighbors for $n\geq 4503$. Thus, we have 
$$\delta(\overline{G}[B])\geq \max\left\{\frac{|B|+2}{3},\alpha(\overline{G}[B]) \right\},$$
and consequently by \Cref{Thm:Nash-william}, $\overline{G}[B]$ is Hamiltonian.

Recall that $|Y\setminus Y'|=2001$. Now, we prove that there exists a $y\in Y\setminus Y'$ such that $y\sim a$ and $y\sim b$ for some $a\in A$ and $b\in B$. If not, let us assume that for every $y\in Y\setminus Y'$, $y$ is either adjacent to vertices of $A$ only or the vertices of $B$ only, outside $Y\setminus Y'$. Therefore, we can partition $Y\setminus Y'$ as $Y\setminus Y'=Y_A\sqcup Y_B$, where $Y_A$ (respectively, $Y_B$) is the set of vertices in $Y\setminus Y'$ that are adjacent to vertices of $A$ (respectively, $B$) only, outside of $Y\setminus Y'$.

If $|Y_B|\geq 1001-k$, then for all $a,a'\in A$ we have 
$$|N_{G}(a)\cap N_{G}(a')|\geq |B|+|Y_B|\geq n,$$
which is a contradiction, so  $|Y_B|\leq 1000-k$. However, in this case, $|Y_A|\geq 1001+k$ and consequently for all $b,b'\in B$ we have
$$|N_{G}(b)\cap N_{G}(b')|\geq |A|+|Y_A|\geq n,$$
again a contradiction. Therefore there exists a $y\in Y\setminus Y'$ such that $y\sim a$ and $y\sim b$ for some $a\in A$ and $b\in B$.
\begin{figure}
    \centering
   \begin{tikzpicture}[x=0.8cm,y=0.8cm]
	\begin{pgfonlayer}{nodelayer}
		\node [style=none] (0) at (-2, 2) {};
		\node [style=none] (1) at (0.5, 2.75) {};
		\node [style=none] (2) at (1, 0.25) {};
		\node [style=none] (3) at (-1.5, -0.75) {};
		\node [style=none] (4) at (7.75, 2.75) {};
		\node [style=none] (5) at (4.5, 3.5) {};
		\node [style=none] (6) at (3.5, 1) {};
		\node [style=none] (7) at (7.25, -0.25) {};
		\node [style=none] (8) at (4.5, -2.25) {};
		\node [style=none] (9) at (4.5, -4) {};
		\node [style=none] (10) at (0.5, -2.25) {};
		\node [style=none] (11) at (0.5, -4) {};
		%\node [style=none] (12) at (5.25, -3.25) {2000 vertices};
		\node [style=none] (13) at (2.5, -4.5) {$Y\setminus Y'$};
		\node [style=none] (14) at (-0.5, 3.4) {$n-1001-k$};
		\node [style=none] (15) at (6, 4.25) {$n-1001+k$};
		\node [style=none] (16) at (-3.25, 0.75) {$\overline{G}[A]$};
		\node [style=none] (17) at (9, 1) {$\overline{G}[B]$};
		\node [style=Doty,scale=0.65] (18) at (5, 1.75) {};
		\node [style=Doty,scale=0.65] (19) at (6, 0.25) {};
		\node [style=Doty,scale=0.65] (20) at (0, 1) {};
		\node [style=Doty,scale=0.65] (21) at (-1, -0.25) {};
		\node [style=Doty,scale=0.65] (22) at (2.5, -2.75) {};
		\node [style=Doty,scale=0.65] (23) at (2.3, -3.5) {};
		\node [style=none] (24) at (4.5, 1.75) {$b$};
		\node [style=none] (25) at (6.5, 0) {$b'$};
		\node [style=none] (26) at (0.5, 1) {$a$};
		\node [style=none] (27) at (-1.5, -0.25) {$a'$};
		\node [style=none] (28) at (2, -2.75) {$y$};
		\node [style=none] (29) at (2.8, -3.5) {$y'$};
		\node [style=none] (30) at (-0.5, 1.75) {};
		\node [style=none] (31) at (-1.75, 1.25) {};
		\node [style=none] (32) at (5.25, 3) {};
		\node [style=none] (33) at (6.75, 2.5) {};
		\node [style=none] (34) at (7.25, 1) {};
	\end{pgfonlayer}
	\begin{pgfonlayer}{edgelayer}
		\draw (8.center) to (10.center);
		\draw (10.center) to (11.center);
		\draw (11.center) to (9.center);
		\draw (9.center) to (8.center);
		\draw [bend right=45, looseness=1.25] (6.center) to (7.center);
		\draw [bend right=45, looseness=1.25] (7.center) to (4.center);
		\draw [bend right, looseness=1.25] (4.center) to (5.center);
		\draw [bend right=45] (5.center) to (6.center);
		\draw [bend left=45] (1.center) to (2.center);
		\draw [bend left=45] (2.center) to (3.center);
		\draw [bend left=45, looseness=1.50] (3.center) to (0.center);
		\draw [bend left=45] (0.center) to (1.center);
		\draw [line width=1.1pt](18) to (22);
		\draw [line width=1.1pt](22) to (20);
		\draw [line width=1.1pt](19) to (23);
		\draw [line width=1.1pt](23) to (21);
		\draw [in=-90, out=75, looseness=1.75,line width=1.1pt,dashed] (21) to (31.center);
		\draw [in=-120, out=75, looseness=2.00,line width=1.1pt,dashed] (31.center) to (30.center);
		\draw [bend left=105, looseness=2.25,line width=1.1pt,dashed] (30.center) to (20);
		\draw [in=-150, out=45, looseness=2.00,line width=1.1pt,dashed] (18) to (32.center);
		\draw [in=-135, out=30, looseness=1.50,line width=1.1pt,dashed] (32.center) to (33.center);
		\draw [in=60, out=45,line width=1.1pt,dashed] (33.center) to (34.center);
		\draw [in=75, out=-120, looseness=2.25,line width=1.1pt,dashed] (34.center) to (19);
	\end{pgfonlayer}
\end{tikzpicture}

    \caption{Existence of $C_m$ in $\overline{G}$ when $\overline{G}[X']$ is disconnected}
    \label{fig:existenceOfCycle0}
\end{figure}
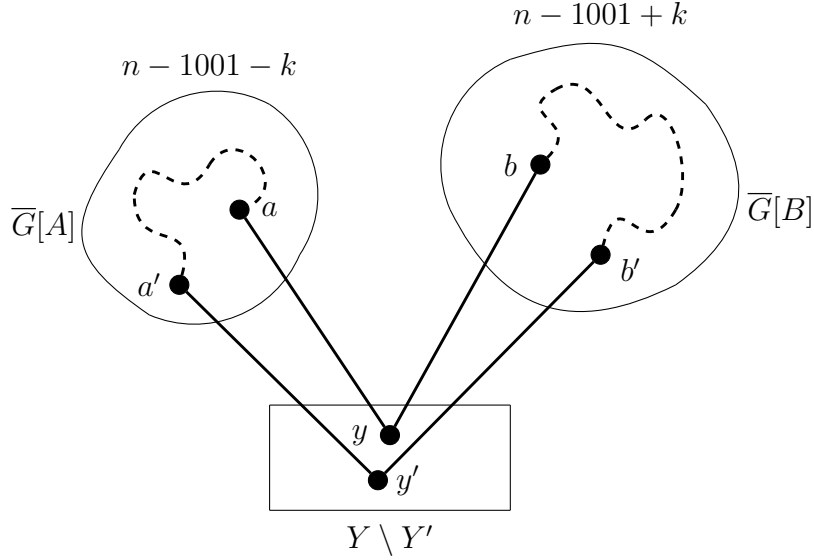

Now we claim that there exists a vertex $y'\in Y\setminus Y'$ such that $y'\neq y$, $y'\sim a'$ and $y'\sim b'$ for some $a'\in A\setminus \{a\}$ and $b'\in B\setminus \{b\}$ (see \Cref{fig:existenceOfCycle0}). Note that for any two vertices $b_1,b_2\in B$, we have 
\begin{equation}\label{eqn:b1b2_nonneighbor_in_Z}
    |\{N_{\overline{G}}(b_1) \cup N_{\overline{G}}(b_2)\}\cap Y\setminus Y'|\geq 1001-k,
\end{equation}
otherwise $|N_G(b_1)\cap N_G(b_2)|\geq n$. Now as $\overline{G}[B]$ is Hamiltonian, choose two vertices $b_1$ and $b_2$ on a Hamiltonian cycle such that the length of the path from $b$ to $b_1$ along that cycle in the clockwise direction is equal to the length of the path from $b$ to $b_2$ along the cycle in the counterclockwise direction. By \Cref{eqn:b1b2_nonneighbor_in_Z}, we have 

\begin{equation}\label{eqn:b1b2}
    |\{N_{\overline{G}}(b_1) \cup N_{\overline{G}}(b_2)\}\cap \{(Y\setminus Y')\setminus \{y\}\}|\geq 1000-k.
\end{equation}

Set $Z=\{N_{\overline{G}}(b_1) \cup N_{\overline{G}}(b_2)\}\cap \{(Y\setminus Y')\setminus \{y\}\}$. Now we claim that there exists $y'\in Z$ such that $y'\sim a'$ for some $a'\in A\setminus\{a\}$. If not, then for all  $a'\in A\setminus\{a\}$, $a'\nsim y'$ for all $y'\in Z$. In this case, there exist $a_1,a_2\in A\setminus\{a\}$ such that  
$$|N_G(a_1)\cap N_G(a_2)|\geq |B|+|Z|+|\{v\}|\geq n,$$
a contradiction. Therefore there exists $y'\in Z$ such that $y'\sim a'$ for some $a'\in A\setminus \{a\}$. But $y'$ is adjacent to at least one of $b_1$ and $b_2$, call that adjacent vertex $b'$. 

Note that, by the way we have chosen $b'$, for every $j \in \{2,\ldots,n-1001+k\}$, we can select $b'$ so that there exists a path of order $j$ between $b$ and $b'$. Moreover, since $\overline{G}[A]$ is panconnected, for any  $a,a'\in A$ there exists an $a-a'$ path of every order from 3 to $n-1001-k$. Therefore we can obtain a cycle $C_m=a\,y\,b\,\ldots b'\,y'\,a'\,\ldots a$ in $\overline{G}$ for every $n+1\leq m\leq 2n-2000$ by choosing paths of suitable order from $a$ to $a'$ in $\overline{G}[A]$ and $b$ to $b'$ in $\overline{G}[B]$, which is a contradiction. Therefore $\overline{G}[X']$ in not disconnected. 

\textbf{Case 3.1(b).} Assume that $\kappa(\overline{G}[X'])=1$, that is, there exists a cut vertex, say $w$. Therefore, we can assume $X'\setminus \{w\}=A\sqcup B$, where
\begin{eqnarray*}
     n-2001 \leq &|A|&\leq n-1002\\
    n-1001\leq &|B|&\leq n-2,
\end{eqnarray*}
Again, we consider 
\begin{eqnarray*}
    |A|&=&n-1002-k,\\
    |B|&=&n-1001+k, \;\;\;\; \text{ for some } 0\leq k \leq 999.
\end{eqnarray*}

Following the same arguments as in the previous case, we have that $\overline{G}[A]$ is panconnected and $\overline{G}[B]$ is Hamiltonian for $n \geq 4507$. Now consider $w\sim a$ and $w\sim b$ for some $a\in A$ and $b\in B$. Similar to the previous case, we can show that there exists a vertex $a'\in A\setminus \{a\}$ and we can choose $b'\in B\setminus \{b\}$  such that for every $j \in \{2,\ldots,n-1001+k\}$, there exists a path of order $j$ between $b$ and $b'$ in $\overline{G}[B]$, and $a'\sim y$ and $b'\sim y$, in $\overline{G}$, for some $y\in Y\setminus Y'$ (see \Cref{fig:existenceOfCycle1}). Since $\overline{G}[A]$ is panconnected, therefore, we can obtain a cycle $C_m: w\, b\, \ldots, b'\,y\,a'\,\ldots a\,w$ in $\overline{G}$, for every $m\in \{n+1, 2n-2001\}$, which is a contradiction. Therefore $\kappa(\overline{G}[X'])\neq 1$.
\begin{figure}[!ht]
    \centering
    \begin{tikzpicture}[x=0.8 cm,y=0.8 cm]
	\begin{pgfonlayer}{nodelayer}
		\node [style=none] (0) at (-2, 2) {};
		\node [style=none] (1) at (0.5, 2.75) {};
		\node [style=none] (2) at (1, 0.25) {};
		\node [style=none] (3) at (-1.5, -0.75) {};
		\node [style=none] (8) at (4.5, -2.25) {};
		\node [style=none] (9) at (4.5, -4) {};
		\node [style=none] (10) at (0.5, -2.25) {};
		\node [style=none] (11) at (0.5, -4) {};
		\node [style=none] (13) at (2.5, -4.5) {$Y\setminus Y'$};
		\node [style=none] (14) at (-0.5, 3.4) {$n-1002-k$};
		\node [style=none] (15) at (5.5, 3.75) {$n-1001+k$};
		\node [style=none] (16) at (-3.25, 0.75) {$\overline{G}[A]$};
		\node [style=none] (17) at (8, 1) {$\overline{G}[B]$};
		\node [style=Doty, scale=0.65] (18) at (3.75, 1.5) {};
		\node [style=Doty, scale=0.65] (19) at (3.9, 0.75) {};
		\node [style=Doty, scale=0.65] (20) at (0.5, 1.75) {};
		\node [style=Doty, scale=0.65] (21) at (-1, -0.25) {};
		\node [style=Doty, scale=0.65] (23) at (2.3, -3.5) {};
		\node [style=none] (24) at (4.2, 0.75) {$b'$};
		\node [style=none] (25) at (4.1, 1.5) {$b$};
		\node [style=none] (26) at (-0.5, -0.25) {$a'$};
		\node [style=none] (27) at (0.25, 1.25) {$a$};
		\node [style=none] (28) at (3, -3.5) {$y$};
		\node [style=none] (30) at (-1, 1.5) {};
		\node [style=none] (31) at (-1.75, 1) {};
		\node [style=none] (32) at (5.5, 3.25) {};
		\node [style=none] (33) at (7.25, 1.5) {};
		\node [style=Doty] (35) at (2.25, 3.25) {};
		\node [style=none] (36) at (2.25, 3.75) {$w$};
		\node [style=none] (37) at (5.5, -0.25) {};
	\end{pgfonlayer}
	\begin{pgfonlayer}{edgelayer}
		\draw (8.center) to (10.center);
		\draw (10.center) to (11.center);
		\draw (11.center) to (9.center);
		\draw (9.center) to (8.center);
		\draw [in=75, out=15] (1.center) to (2.center);
		\draw [bend left=45] (2.center) to (3.center);
		\draw [bend left=45, looseness=1.50] (3.center) to (0.center);
		\draw [in=-165, out=60, looseness=1.50] (0.center) to (1.center);
		\draw [line width=1.1pt] (19) to (23);
		\draw [line width=1.1pt] (23) to (21);
		\draw [line width=1.1pt, dashed, in=-90, out=90, looseness=1.25] (21) to (31.center);
		\draw [line width=1.1pt, dashed, in=-120, out=75, looseness=2.00] (31.center) to (30.center);
		\draw [line width=1.1pt, dashed, bend left=60, looseness=1.50] (30.center) to (20);
		\draw [line width=1.1pt](18) to (35);
		\draw [line width=1.1pt](35) to (20);
		\draw [line width=1.1pt, dashed,bend left=315] (32.center) to (18);
		\draw [line width=1.1pt, dashed,bend right=45] (18) to (37.center);
		\draw [line width=1.1pt, dashed,bend right=45] (37.center) to (33.center);
		\draw [line width=1.1pt, dashed,bend right=45] (33.center) to (32.center);
	\end{pgfonlayer}
\end{tikzpicture}

    \caption{Existence of $C_m$ in $\overline{G}$ when $\kappa(\overline{G}[X'])=1$}
    \label{fig:existenceOfCycle1}
\end{figure}
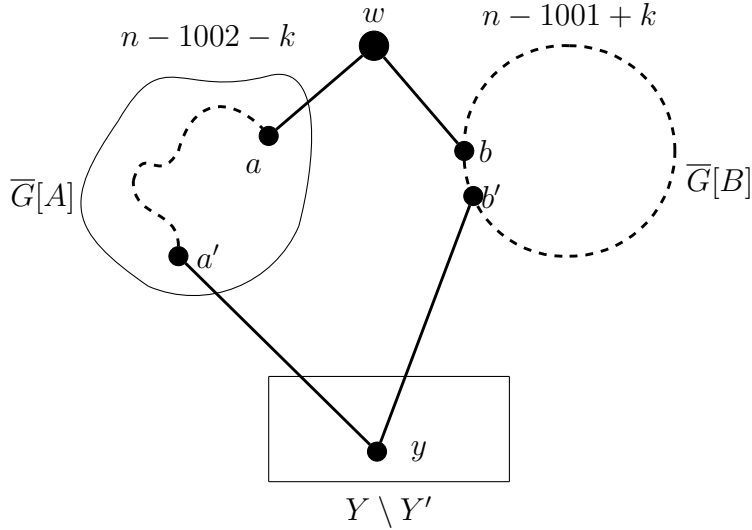

\textbf{Case 3.1(c).} Assume that $\kappa(\overline{G}[X'])=2$, that is, there exist two vertices $w, w'\in X'$ such that they form a cut set for $\overline{G}[X']$. Assume $X'\setminus\{w,w'\}=A\sqcup B$, where
\begin{eqnarray*}
     n-2002 \leq &|A|&\leq n-1002\\
    n-1002\leq &|B|&\leq n-2,
\end{eqnarray*}
Similar to the previous cases, we consider 
\begin{eqnarray*}
    |A|&=&n-1002-k,\\
    |B|&=&n-1002+k, \;\;\;\; \text{ for some } 0\leq k \leq 1000.
\end{eqnarray*}

Again, similar to the previous cases, we have that $\overline{G}[A]$ is panconnected and $\overline{G}[B]$ is Hamiltonian  for $n \geq 4511$. Since $\overline{G}[X']$ is 2-connected, there exist vertices $a,a'\in A$ and $b,b'\in B$ such that $w'$ is adjacent to $a'$ and $b'$, and $w$ is adjacent to $a$ and $b$ (see \Cref{fig:existenceOfCycle2}). We choose such $b$ and $b'$ so that their cyclic distance on a Hamiltonian cycle in $\overline{G}[B]$ is minimum. Let $P_1$ and $P_2$ denote the two $b-b'$ subpaths on that cycle, where $P_1$ is the shorter path and $P_2$ is the longer one.

\textbf{Sub-case 3.1(c)(\rm{I}).} Suppose $4 \le |V(P_1)| \le 2008$. Let $P_3 = wP_2w'$. Then the order of $P_3$ satisfies
\[
n - 3006 + k \le |V(P_3)| \le n - 1002 + k,
\]
and hence we may parametrize it as
\begin{equation}\label{eqn:P3lenth}
|V(P_3)| = n - 3006 + k + s,
\end{equation}
where $s \in [0,2004]$ and $k \in [0,1000]$.

Since $\overline{G}[A]$ is panconnected and has diameter at most $2$, we can find an $a$--$a'$ path $P_4$ in $\overline{G}[A]$ of any prescribed order between $3$ and $n - 1002 - k$. In particular, we choose
\begin{equation}\label{eqn:P4lenth}
|V(P_4)| = (n - 1002 - k) - s - t,
\end{equation}
where $t \in [0, n - 4009]$.

Now joining $P_3$ and $P_4$ using the edges $a'w'$ and $aw$, we obtain a cycle $C_m$ (see \Cref{fig:existenceOfCycle2}$(\mathrm{I})$). By construction,
\[
m = |V(P_3)| + |V(P_4)|,
\]
and thus by choosing appropriate $t$, we have $\overline{G}\supseteq C_m$ for every  $m\in [n+1, 2n - 4008]$, a contradiction.

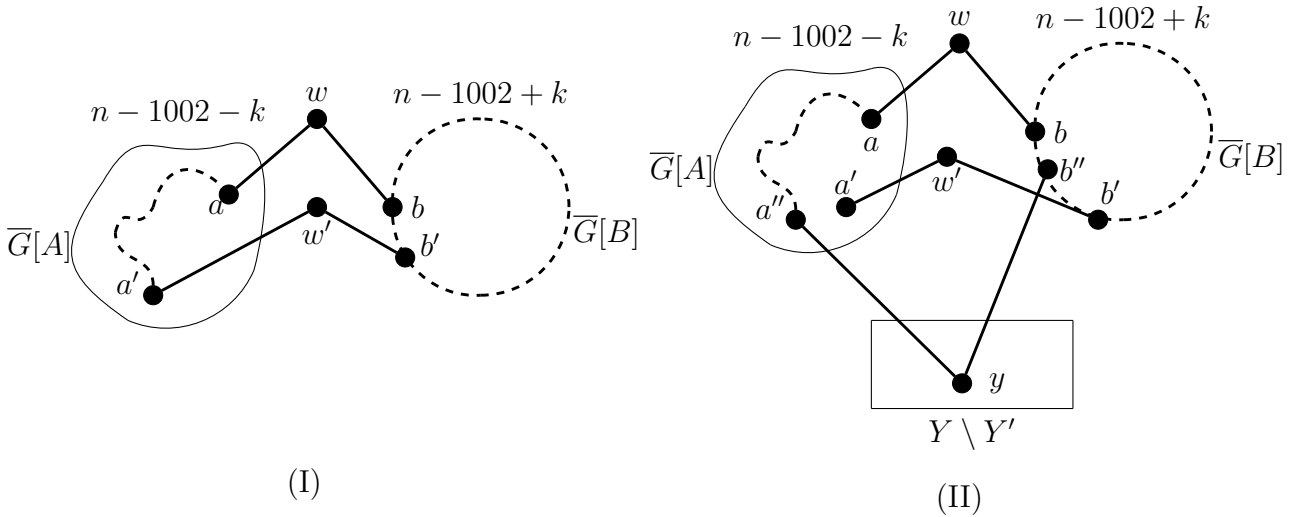
\begin{figure}[!ht]
    \centering
   \begin{tikzpicture}[x=0.67 cm,y=0.67cm]
	\begin{pgfonlayer}{nodelayer}
		\node [style=none] (0) at (-0.5, 0.5) {};
		\node [style=none] (1) at (2, 1.25) {};
		\node [style=none] (2) at (2.5, -1.25) {};
		\node [style=none] (3) at (0, -2.25) {};
		\node [style=none] (14) at (1, 1.9) {$n-1002-k$};
		\node [style=none] (15) at (7, 2.25) {$n-1002+k$};
		\node [style=none] (16) at (-1.75, -0.75) {$\overline{G}[A]$};
		\node [style=none] (17) at (9.5, -0.5) {$\overline{G}[B]$};
		\node [style=Doty, scale=0.65] (18) at (5.25, 0) {};
		\node [style=Doty, scale=0.65] (19) at (5.5, -1) {};
		\node [style=Doty, scale=0.65] (20) at (2, 0.25) {};
		\node [style=Doty, scale=0.65] (21) at (0.5, -1.75) {};
		\node [style=none] (24) at (6, -0.75) {$b'$};
		\node [style=none] (25) at (5.75, 0) {$b$};
		\node [style=none] (26) at (0, -1.5) {$a'$};
		\node [style=none] (27) at (1.75, 0) {$a$};
		\node [style=none] (30) at (0.5, 0) {};
		\node [style=none] (31) at (-0.25, -0.5) {};
		\node [style=none] (32) at (7, 1.75) {};
		\node [style=none] (33) at (8.75, 0) {};
		\node [style=Doty, scale=0.65] (35) at (3.75, 1.75) {};
		\node [style=none] (36) at (3.75, 2.25) {$w$};
		\node [style=none] (37) at (7, -1.75) {};
		\node [style=none] (38) at (12.25, 2) {};
		\node [style=none] (39) at (14.75, 2.75) {};
		\node [style=none] (40) at (15.25, 0.25) {};
		\node [style=none] (41) at (12.75, -0.75) {};
		\node [style=none] (42) at (18.75, -2.25) {};
		\node [style=none] (43) at (18.75, -4) {};
		\node [style=none] (44) at (14.75, -2.25) {};
		\node [style=none] (45) at (14.75, -4) {};
		\node [style=none] (46) at (16.75, -4.5) {$Y\setminus Y'$};
		\node [style=none] (47) at (13.75, 3.4) {$n-1002-k$};
		\node [style=none] (48) at (19.75, 3.75) {$n-1002+k$};
		\node [style=none] (49) at (11, 0.75) {$\overline{G}[A]$};
		\node [style=none] (50) at (22.3, 1) {$\overline{G}[B]$};
		\node [style=Doty, scale=0.65] (51) at (18, 1.5) {};
		\node [style=Doty, scale=0.65] (52) at (18.25, 0.75) {};
		\node [style=Doty, scale=0.65] (53) at (14.75, 1.75) {};
		\node [style=Doty, scale=0.65] (54) at (13.25, -0.25) {};
		\node [style=Doty, scale=0.65] (55) at (16.55, -3.5) {};
		\node [style=none] (56) at (18.75, 0.75) {$b''$};
        \node [style=none] (78) at (19.5, 0.3) {$b'$};
		\node [style=none] (57) at (18.5, 1.5) {$b$};
		\node [style=none] (58) at (14.25, 0.5) {$a'$};
		\node [style=none] (59) at (14.75, 1.25) {$a$};
		\node [style=none] (60) at (17.25, -3.5) {$y$};
		\node [style=none] (61) at (13.25, 1.5) {};
		\node [style=none] (62) at (12.5, 1) {};
		\node [style=none] (63) at (19.75, 3.25) {};
		\node [style=none] (64) at (21.5, 1.5) {};
		\node [style=Doty, scale=0.65] (65) at (16.5, 3.25) {};
		\node [style=none] (66) at (16.5, 3.75) {$w$};
		\node [style=none] (67) at (19.75, -0.25) {};
		\node [style=Doty, scale=0.65] (68) at (3.75, 0) {};
		\node [style=none] (69) at (3.75, -0.5) {};
		\node [style=none] (70) at (3.75, -0.5) {$w'$};
		\node [style=Doty, scale=0.65] (71) at (16.25, 1) {};
		\node [style=Doty, scale=0.65] (72) at (19.25, -0.25) {};
		\node [style=Doty, scale=0.65] (73) at (14.25, 0) {};
		\node [style=none] (74) at (3.5, -5.5) {$(\mathrm{I})$};
		\node [style=none] (75) at (16.5, -5.8) {$(\mathrm{II})$};
		\node [style=none] (76) at (16.25, 0.56) {$w'$};
		\node [style=none] (77) at (12.75, 0) {$a''$};
	\end{pgfonlayer}
	\begin{pgfonlayer}{edgelayer}
		\draw [in=75, out=15] (1.center) to (2.center);
		\draw [bend left=45] (2.center) to (3.center);
		\draw [bend left=45, looseness=1.50] (3.center) to (0.center);
		\draw [in=-165, out=60, looseness=1.50] (0.center) to (1.center);
		\draw [line width=1.1pt, dashed, in=-90, out=90, looseness=1.25] (21) to (31.center);
		\draw [line width=1.1pt, dashed, in=-120, out=75, looseness=2.00] (31.center) to (30.center);
		\draw [line width=1.1pt, dashed, bend left=60, looseness=1.50] (30.center) to (20);
		\draw [line width=1.1pt](18) to (35);
		\draw [line width=1.1pt](35) to (20);
		\draw [line width=1.1pt, dashed,bend left=315] (32.center) to (18);
		\draw [line width=1.1pt, dashed,bend right=45] (18) to (37.center);
		\draw [line width=1.1pt, dashed,bend right=45] (37.center) to (33.center);
		\draw [line width=1.1pt, dashed,bend right=45] (33.center) to (32.center);
		\draw (42.center) to (44.center);
		\draw (44.center) to (45.center);
		\draw (45.center) to (43.center);
		\draw (43.center) to (42.center);
		\draw [in=75, out=15] (39.center) to (40.center);
		\draw [bend left=45] (40.center) to (41.center);
		\draw [bend left=45, looseness=1.50] (41.center) to (38.center);
		\draw [in=-165, out=60, looseness=1.50] (38.center) to (39.center);
		\draw [line width=1.1pt] (52) to (55);
		\draw [line width=1.1pt] (55) to (54);
		\draw [line width=1.1pt, dashed, in=-90, out=90, looseness=1.25] (54) to (62.center);
		\draw [line width=1.1pt, dashed, in=-120, out=75, looseness=2.00] (62.center) to (61.center);
		\draw [line width=1.1pt, dashed, bend left=60, looseness=1.50] (61.center) to (53);
		\draw [line width=1.1pt](51) to (65);
		\draw [line width=1.1pt](65) to (53);
		\draw [line width=1.1pt, dashed,bend left=315] (63.center) to (51);
		\draw [line width=1.1pt, dashed,bend right=45] (51) to (67.center);
		\draw [line width=1.1pt, dashed,bend right=45] (67.center) to (64.center);
		\draw [line width=1.1pt, dashed,bend right=45] (64.center) to (63.center);
		\draw [line width=1.1pt](19) to (68);
		\draw [line width=1.1pt](68) to (21);
		\draw [line width=1.1pt](71) to (72);
		\draw [line width=1.1pt](71) to (73);
	\end{pgfonlayer}
\end{tikzpicture}

    \caption{Existence of $C_m$ in $\overline{G}$ when $\kappa(\overline{G}[X'])=2$}
    \label{fig:existenceOfCycle2}
\end{figure}

\textbf{Sub-case 3.1(c)(\rm{II}).} Now suppose $|V(P_1)|\geq 2009$. Let $b_1$ and $b_2$ be two vertices on the path $P_1$ such that $b_1$ has distance three from $b$ and $b_2$ has distance  three from $b'$ along the Hamiltonian cycle. Clearly, both the $b_1$ and $b_2$ are non-adjacent to both the $w$ and $w'$ in $\overline{G}$. Note that 
\begin{equation}
    |\{N_{\overline{G}}(b_1)\cup N_{\overline{G}}(b_2)\}\cap \{Y\setminus Y'\}|\geq 1002-k,
\end{equation}
otherwise, $|N_G(b_1)\cap N_G(b_2)\cap Y\setminus Y'|\geq 1000+k$, and this implies $|N_G(b_1)\cap N_G(b_2)|\geq 1000+k + |A|+|\{w,w'\}|\geq n$, a contradiction.

Now, set $Z=\{N_{\overline{G}}(b_1)\cup N_{\overline{G}}(b_2)\}\cap \{Y\setminus Y'\}$. We claim that there exists a vertex $a''\in A \setminus \{a,a'\}$ such that $a''\sim y$, for some $y\in Z$. If not, for every $a_1,a_2\in A\setminus \{a,a'\}$, we have 
$ |N_G(a_1)\cap N_G(a_2)|\geq |B|+|Z|\geq n$, again a contradiction. Therefore there exists a vertex $a''\in A \setminus \{a,a'\}$ such that $a''\sim y$, for some $y\in Z$. Note that $y$ is adjacent to at least one of $b_1$ or $b_2$. Without loss of generality, we assume $y\sim b_1$ and call $b_1$ as $b''$ (see \Cref{fig:existenceOfCycle2}$(\mathrm{II})$).

Now we treat $a''$, $b''$, and $y$ as playing the roles of $a'$, $b'$, and $w'$, respectively, as in {Sub-case 3.1(c)(\rm{I})}. Proceeding as in {Sub-case 3.1(c)(\rm{I})}, we obtain a contradiction.

\textbf{Sub-case 3.1(c)(\rm{III}).} Consider the case $|V(P_1)| \in \{2,3\}$. Let $b_1$ and $b_2$ be the vertices on $P_2$ at distance three from $b$ and $b'$, respectively, along the Hamiltonian cycle. If either $b_1$ or $b_2$ is adjacent to $w$ or $w'$, denote such a vertex by $b''$. Depending on whether $b''$ is adjacent to $w$ or $w'$, we treat it as $b$ or $b'$, respectively, and proceed as in {Sub-case 3.1(c)(\rm{I})}. Otherwise, if neither $b_1$ nor $b_2$ is adjacent to $w$ or $w'$, we proceed as in {Sub-case 3.1(c)(\rm{II})}, followed by {Sub-case 3.1(c)(\rm{I})} and get a contradiction. Therefore $\kappa(\overline{G}[X'])\neq 2$.  This completes the proof of \Cref{sub:claim:2a:thm1}.

Since $G\nsupseteq K_{2,n}$, for any $x_1,x_2\in X'$, we have
\begin{equation}\label{Eqn:n-2001}
    |N_{\overline{G}[X']}(x_1) \cup N_{\overline{G}[X']}(x_2)|\geq 2n-2002-2-n+1=n-2003.
\end{equation}
Since $\overline{G}[X']$ is $3$-connected, thus by the \Cref{Thm:circumference_3connectedGraphs} and \Cref{Eqn:n-2001}, we have 
\begin{equation}\label{Eqn:3n/2-4002}
    c(\overline{G}[X'])\geq \frac{3}{2}(n-2003).
\end{equation}
Applying \Cref{lem:cycle_lemma} to $\overline{G}[X']$, which is $2$-connected and satisfies \Cref{Eqn:n-2001,Eqn:3n/2-4002}, we obtain
\begin{equation}
    c(\overline{G}[X']) \geq 2n - 4008.
\end{equation}

 Now we show that $\overline{G}[X']$ is not bipartite. Suppose $\overline{G}[X']$ is bipartite with bipartition $X'=A\sqcup B$. Also, we know that 
$$\delta(\overline{G}[X'])\geq \bigg\lfloor \frac{n}{2}\bigg\rfloor -250\geq \frac{n+1}{3}+4\geq \frac{\max\{|A|,|B|\}}{3}+4,$$ for all  $n\geq 1530$. By \Cref{Thm:Bi-pancyclic_Hu}, $\overline{G}[X']$ is weakly bi-pancyclic graph of girth 4. Since  $c(\overline{G}[X'])\geq 2n-4008$, therefore $\overline{G}[X']\supseteq C_m$ for all even $m\in \{n+1,\ldots, 2n-4008\}$, which is a contradiction. 

Now, we have $\overline{G}[X']$ is a $2$-connected, non-bipartite graph on $2n-2002$ vertices with the  circumference $c(\overline{G}[X'])\geq 2n-4008$ such that 
$$\delta(\overline{G}[X'])\geq \frac{|X'|}{4}+250.$$
Therefore by \Cref{Thm:Weakpan-Brandt_nby4_250}, $\overline{G}[X']\supseteq C_m$ for all $m\in \{n+1,\ldots, 2n-4008\}$, again a contradiction. This completes the proof of our result.

\end{proof}

The proof of the following theorem shares several ideas with that of the \Cref{thm:main_result11}, but also requires a number of additional arguments. For this reason, we provide a separate and shorter proof, highlighting only the key differences.

\begin{theorem}\label{thm:main_result2}
    For even $n\geq 4516$, we have $$R(K_{2,n},C_n)=2n-1.$$
\end{theorem}
\begin{proof}
For the lower bound, consider the graph $G=K_{n-1,n-1}$. Clearly, $G\nsupseteq K_{2,n}$ and $\overline{G}\nsupseteq C_n$, and therefore $R(K_{2,n},C_n)\geq 2n-1$.
    Let $G$ be a graph on $2n-1$ vertices such that $G\nsupseteq K_{2,n}$ and $\overline{G}\nsupseteq C_n$. Similar to the proof of \Cref{thm:main_result11}, we, again, divide the proof into two cases.

\textbf{Case 1. }$\delta(\overline{G})\geq \frac{n}{2}+250$.

First we prove that $\overline{G}$ is $2$-connected. If not, assume that $\kappa(\overline{G})=0$, therefore we can partition $V(\overline{G})$ as $V(\overline{G})= A \sqcup B$, where $A$ is the set of vertices of the smallest connected component of $\overline{G}$. Since $|A|\geq \delta(\overline{G})+1>2$ and $|B|\geq n$, therefore we get a $K_{2,n}$ in $G$, a contradiction. Therefore $\kappa(\overline{G})\neq 0$.

     Now assume $\kappa(\overline{G})= 1$. Let $w$ be a cut vertex and $V(\overline{G})\setminus \{w\}=A \sqcup B$. If  $|B|\geq n$ then $G\supseteq K_{2,n}$, therefore the only possible case is $|A|=n-1$ and $|B|=n-1$. Note that $w$ is adjacent to at least $n-2$ vertices in both $A$ and $B$ in $\overline{G}$, otherwise we can find two vertices either from $A$ or $B$, say $A$, not adjacent to $w$ and $n-1$ vertices in $B$, thus they have $n$ common non-neighbor in $\overline{G}$, a contradiction.  Note that any vertex in $A$ (similarly in $B$) can not have two non-neighbors in $\overline{G}[A]$ (similarly in $\overline{G}[B]$) otherwise $G\supseteq K_{2,n}$, therefore $\delta(\overline{G}[A])\geq n-3$ and $\delta(\overline{G}[B])\geq n-3$. By \Cref{Thm:pancyclic-Bondy1971}, both $\overline{G}[A\sqcup \{w\}]$ and $\overline{G}[B\sqcup \{w\}]$ are pancyclic. This implies $\overline{G}$ contains $C_n$, a contradiction. Therefore $\kappa(\overline{G})\neq 1$.

 Since $\overline{G}$ is $2$-connected and $\delta(\overline{G})\geq n/2+250> {|V(\overline{G})|}/{4}+250$, from \Cref{Dirac}, we have that
\begin{equation}\label{eqn:circumference_G_bar}
c(\overline{G})\geq n+500 >n.    
\end{equation}
 Similar to the proof of the previous theorem, we can assume $\overline{G}$ is non-bipartite. Using the $2$-connectedness and \Cref{eqn:circumference_G_bar}, we conclude from \Cref{Thm:Weakpan-Brandt_nby4_250} that $\overline{G}$ contains a $C_{n}$, a contradiction. Hence this case is not possible.

 \textbf{Case 2. }$\delta(\overline{G})\leq \frac{n}{2}+249$.

In this case, we have $\Delta(G)\geq 3n/2-251$. So we can choose a vertex $v$ such that $|N_G(v)|\geq 3n/2-251$ and take  $X\subseteq N_G(v)$ such that $|X|=3n/2-251$. Let $Y:=V(G)\setminus (X\cup \{v\})$.
Note that for all $x\in X$, we have $|N_{G}(x)\cap X|\leq n-1$, this implies 
\begin{equation}\label{eqn:x-degree}
    |N_{\overline{G}}(x)\cap X|\geq n/2-251\;\;\;\;\text{for all }x\in X.
\end{equation}
 Similarly 
 \begin{equation}\label{eqn:y-degree}
     |N_{\overline{G}}(y)\cap X|\geq n/2-250 \;\;\;\;\text{for all }y\in Y.
 \end{equation}
 Choose $Y'\subseteq Y$ such that $|Y'|=n/2-1753$ and take $X'=X\sqcup Y'$. Now consider the graph $\overline{G}[X']$ on $2n-2004$ vertices. By \Cref{eqn:x-degree} and \Cref{eqn:y-degree}, we have 
 \begin{equation}\label{eqn:min_degree_Comp_G[X']_thm2}
     \delta(\overline{G}[X'])\geq \frac{n}{2}-251=\frac{|X'|}{4}+250.
 \end{equation}

\begin{claim}\label{sub:claim:2a:thm2}
     $\overline{G}[X']$ is $3$-connected.
 \end{claim}
\emph{Proof of $\Cref{sub:claim:2a:thm2}$.} \textbf{Case 4.1(a).} If not, let $\overline{G}[X']$ is disconnected. Due to a similar reason used in the \Cref{sub:claim:2a:thm1}, $\overline{G}[X']$ can not have more than two connected components and if   $X'=A\sqcup B$ with $|A|\leq |B|$ then we have
\begin{eqnarray*}
    |A|&=&n-1002-k,\\
    |B|&=&n-1002+k, \;\;\;\; \text{ for some } 0\leq k \leq 1000.
\end{eqnarray*}
Again, using similar arguments as in the \Cref{sub:claim:2a:thm1}, we can conclude that $\overline{G}[A]$ is panconnected and $\overline{G}[B]$ is $2$-connected and consequently Hamiltonian for $n\geq 4508$. Again similar to the \Cref{sub:claim:2a:thm1}, there exist $y, y' \in Y \setminus Y'$ such that $y$ is adjacent to $a$ and $b$, and $y'$ is adjacent to $a'$ and $b'$ for $a,a'\in A$ and $b,b'\in B$. Choosing $b$ and $b'$ appropriately and arguing as before, we obtain a cycle $b \ldots b' y' a' \ldots a y b$ of order $n$ in $\overline{G}$, a contradiction.

\textbf{Case 4.1(b).} Now assume that $\kappa(\overline{G}[X'])=1$, that is, there exists a cut vertex, say $w$. Therefore, we can assume $X'\setminus \{w\}=A\sqcup B$, where
\begin{eqnarray*}
    |A|&=&n-1003-k,\\
    |B|&=&n-1002+k, \;\;\;\; \text{ for some } 0\leq k \leq 1000.
\end{eqnarray*}
Again, similar to the \Cref{sub:claim:2a:thm1}, $\overline{G}[A]$ is panconnected and $\overline{G}[B]$ is Hamiltonian for $n\geq 4512$. Consider $w\sim a$ and $w\sim b$ for some $a\in A, b\in B$. Since $v$ has at least $3n/2-251-|A|$ non-neighbors in $\overline{G}[B]$, hence, we can choose $b_1,b_2\in B\setminus \{b\}$ such that both $b_1, b_2$ are non-adjacent to $v$ in $\overline{G}$ and both $b_1,b_2$ are at distance at least four from $b$ on the Hamiltonian cycle of $\overline{G}[B]$. Therefore, 
\begin{equation*}
    |\{N_{\overline{G}}(b_1)\cup N_{\overline{G}}(b_2)\}\cap \{Y\setminus Y'\}|\geq 1001-k.
\end{equation*}
Set $Z=\{N_{\overline{G}}(b_1)\cup N_{\overline{G}}(b_2)\}\cap \{Y\setminus Y'\}$. In this case, we claim that there exists a vertex $a'\in A\setminus \{a\}$ such that $a'\sim y$ for some $y\in Z$. Otherwise there exists $a_1,a_2\in A\setminus \{a\}$ such that 
\begin{equation*}
    |N_G(a_1)\cap N_G(a_2)|\geq |B|+|Z|+|\{v\}|\geq n.
\end{equation*}
  Without loss of generality, we can assume that $y\sim a'$ and $y\sim b'$, where $b'=b_1$. Now, by choosing the larger path $b-b'$ on the Hamiltonian cycle in $\overline{G}[B]$ and an appropriate path $a-a'$ in $\overline{G}[A]$, we can obtain a cycle $w\,b\,\ldots b'\,y\,a'\,\ldots\, a\,w$ of order $n$ in $\overline{G}$, a contradiction.

\textbf{Case 4.1(c).} Assume that $\kappa(\overline{G}[X'])=2$, that is, there exist two vertices $w, w'\in X'$ such that they form a cut set for $\overline{G}[X']$. Assume $X'\setminus\{w,w'\}=A\sqcup B$, where
\begin{eqnarray*}
    |A|&=&n-1003-k,\\
    |B|&=&n-1003+k, \;\;\;\; \text{ for some } 0\leq k \leq 1001.
\end{eqnarray*}

Further, similar to the \Cref{sub:claim:2a:thm1}, we have that $\overline{G}[A]$ is panconnected and $\overline{G}[B]$ is Hamiltonian for $n\geq 4516$. Since $\overline{G}[X']$ is 2-connected, there exist vertices $a,a'\in A$ and $b,b'\in B$ such that $w'$ is adjacent to $a'$ and $b'$, and $w$ is adjacent to $a$ and $b$. Now, if $b$ and $b'$ have distance at least four along a Hamiltonian cycle in $\overline{G}[B]$, then we obtain a cycle of length $n$ as in the previous case. Otherwise, we choose a vertex $b''\sim b$ or $b''\sim b'$  in $\overline{G}[B]$ such that the path $b' \ldots b''\,b$ or $b \ldots b''\,b'$ has length at least $n/2$ and at most $n-5$, respectively, since $|N_{\overline{G}[B]}(b)\cup N_{\overline{G}[B]}(b')|\geq n-2005$. Therefore, using this path and a path $a-a'$ of suitable order in $\overline{G}[A]$, we can obtain a cycle of order $n$ in $\overline{G}$ as in the previous case, a contradiction. 

This completes the proof of \Cref{sub:claim:2a:thm2}.

Note that $|N_{\overline{G}[X']}(x_1)\cup N_{\overline{G}[X']}(x_2)|\geq n-2005$, for every $x_1,x_2\in X'$. Since $\overline{G}[X']$ is $3$-connected, by   \Cref{Thm:circumference_3connectedGraphs} we have
\begin{equation}\label{eqn:circumference_GX'_thm2}
    c(\overline{G}[X'])\geq \frac{3}{2}(n-2005)>n.
\end{equation}

Suppose $\overline{G}[X']$ is bipartite with bipartition $X'=A\sqcup B$. Also, we know that 
$$\delta(\overline{G}[X'])\geq \frac{n}{2}-251\geq \frac{n+1}{3}+4\geq \frac{\max\{|A|,|B|\}}{3}+4.$$
 By \Cref{Thm:Bi-pancyclic_Hu}, $\overline{G}[X']$ is weakly bi-pancyclic graph of girth 4. Since  $c(\overline{G}[X'])>n$, consequently $\overline{G}[X']\supseteq C_n$, which is a contradiction. Therefore $\overline{G}[X']$ is non-bipartite.

From \Cref{eqn:min_degree_Comp_G[X']_thm2,eqn:circumference_GX'_thm2}, we have $\overline{G}[X']$ is a $2$-connected, non-bipartite graph on $2n-2004$ vertices with circumference strictly greater than $n$ such that 
\begin{equation}\label{eqn:G_bar_min_deg}
    \delta(\overline{G}[X'])\geq \frac{|X'|}{4}+250,
\end{equation}
therefore by \Cref{Thm:Weakpan-Brandt_nby4_250}, $\overline{G}[X']$ is weakly pancyclic graph and consequently $\overline{G}[X']\supseteq C_n$, a contradiction. This completes the proof.
\end{proof}

\begin{acknowledgement}
This research work has no associated data. The work of the author Abisek Dewan is supported by University Grants Commission, India (Beneficiary Code/Flag: BWBDA00147662 U). The author Sayan Gupta thanks NISER Bhubaneswar and Homi Bhabha National Institute (HBNI), Mumbai for funding his PhD fellowship. The work of the author Rajiv Mishra is supported by Council of Scientific $\And$ Industrial Research, India(File number: 09/921(0347)/2021-EMR-I).  
\end{acknowledgement}

\bibliographystyle{siam}
	\bibliography{goodbib}
\end{document}